\documentclass[11pt]{amsart}
\usepackage[margin=3cm]{geometry}
\usepackage{bm}
\usepackage{amsthm}
\usepackage{amssymb}
\usepackage{mathrsfs}
\usepackage{hyperref,xcolor,graphicx}
\usepackage[foot]{amsaddr}

\newcommand{\R}{\mathbb{R}}
\newcommand{\Z}{\mathbb{Z}}
\newcommand{\e}{\epsilon}

\DeclareMathOperator{\supp}{supp}
\DeclareMathOperator{\ind}{ind}

\begin{document}

\begin{abstract} 
In this article, we generalize our previous results \cite{Chen2021mean} joint with Pedro Gaspar to higher dimensions, prove the existence of (infinitely many) eternal weak mean curvature flows in $S^{n+1}$ (for all $n \geq 2$) connecting a Clifford hypersurface $S^1(\sqrt{\frac{1}{n}}) \times S^{n-1} (\sqrt{\frac{n-1}{n}})$ to the equatorial spheres $S^n$.
\end{abstract}

\title[Symmetric mean curvature flow on the n-sphere]{Symmetric mean curvature flow on the n-sphere}
\author{Jingwen Chen\textsuperscript{1}}
\address{\parbox{\linewidth}{\textsuperscript{1} Department of Mathematics, The University of Chicago\\ \smallskip 5734 S University Ave, Chicago IL, 60615}}
\email{jingwch@uchicago.edu}
\maketitle

\newtheorem{thm}{Theorem}
\newtheorem*{thm*}{Theorem}
\newtheorem{lem}{Lemma}
\newtheorem{prop}{Proposition}
\newtheorem{coro}{Corollary}

\theoremstyle{definition}
\newtheorem*{defi}{Definition}
\newtheorem*{ex}{Example}
\newtheorem*{rmk}{Remark}

\section{Introduction}

In the previous paper \cite{Chen2021mean} joint with Pedro Gaspar, we study low energy solutions of the \emph{parabolic Allen-Cahn equation}
\begin{equation} \label{PAC}
\e \, \partial_t u = \epsilon \Delta_g u - \frac{1}{\epsilon} W'(u).\tag{PAC}
\end{equation}
in the $3$-sphere in connection with the \emph{mean curvature flow} (MCF). Here $W$ is a nonnegative double-well potential with two wells at $\pm 1$, such as $W(u)=(1-u^2)^2/4$. We show that there are (weak) solutions to the MCF connecting Clifford tori to equatorial spheres constructed as the singular limit of solutions to \eqref{PAC} in $S^3$, as $\e \downarrow 0$, and study a family of such limit flows. The main theorem we proved is the following:

\begin{thm*} [Chen-Gaspar, \cite{Chen2021mean}]
There exist (infinitely many) eternal weak mean curvature flows $\{\Sigma_t\}$ in $S^3$ connecting a Clifford torus ($S^1(\sqrt{\frac{1}{2}}) \times S^1(\sqrt{\frac{1}{2}}) \subset S^3$) to equatorial spheres $S^2$. These flows are smooth for large $|t|$.    
\end{thm*}

The proof relies on the fact that the Clifford torus is the embedded, non-totally geodesic, minimal surface of least area $2 \pi^2$ in $S^3$, which was proved by Marques-Neves in their resolution of the Willmore conjecture \cite{marques2014min}. The main difficulty to generalize the result to higher dimensions is that the Solomon-Yau conjecture, which states that the area of one of the minimal Clifford hypersurfaces gives the lowest value of area among all non-totally geodesic closed minimal hypersurfaces of $S^{n+1}$, is still open for $n \geq 3$.

Here we give the expression of the Clifford-type minimal hypersurfaces:
\begin{align*}
T_{p,q} = S^p (\sqrt{\frac{p}{n}}) \times S^q (\sqrt{\frac{q}{n}}) \subset S^{n+1},
\end{align*}
where $p,q$ are positive integers, and $n = p+q \geq 2$. $T_{p,q}$ has Morse index $n+3$ (see \cite{perdomo2001low}) and nullity $(p+1)(q+1)$ (see \cite{hsiang1971minimal}).

However, under a symmetry assumption, we have the classification of low area minimal hypersurface. For a compact minimal rotational hypersurface $M$ in $S^{n+1}$, i.e. $M$ is invariant under the group of rotations $SO(n)$ (considered as a subgroup of isometries of $S^{n+1}$), Perdomo-Wei \cite{perdomo2015n} (for $2 \leq n \leq 100$) and Cheng-Wei-Zeng \cite{MR4334413} (for all dimensions) showed that the area of $M$ equals to either the area of the equator $S^n$, or the area of the Clifford hypersurface $T_{1,n-1}$, or greater than $2(1 - \frac{1}{\pi}) \text{Area}(T_{1,n-1})$. \smallskip

From a variational and dynamical viewpoint, the equation \eqref{PAC} may be seen as the (negative) gradient flow of the \emph{Allen-Cahn energy functional}
    \[E_{\epsilon}(u)= \int_{M} \left(\frac{\epsilon}{2} |\nabla_{g} u|^2 + \frac{1}{\epsilon} W(u)\right).\]

In this article, by using the classification of low area compact minimal rotational hypersurfaces, we study $SO(n)$ invariant orbits of this gradient flow in $S^{n+1}$ which connect low energy stationary solutions, and describe the mean curvature flows they originate, as $\e \downarrow 0$. We prove:

\begin{thm} \label{main}
For sufficiently small $\e>0$, there are eternal solutions $\{u_\e\}$ of \eqref{PAC} such that 
    \[u_\e^{-\infty} = \lim_{t \to -\infty} u_\e(\cdot, t) \quad \text{and} \quad u_\e^{+\infty} = \lim_{t \to +\infty} u_\e(\cdot ,t)\]
are the symmetric critical points of $E_\e$ which accumulate on a Clifford hypersurface $T_{1,n-1}$ and on an equatorial sphere $S^n$, respectively, as $\e \downarrow 0$. Furthermore, the limit $\{\mu_t\}_{t \in \R}$ of the associated measures
    \[\mu_{\e,t}:=\left(\frac{\e}{2} \, |\nabla u_\e(\cdot, t)|^2 + \frac{W(u_\e(\cdot, t))}{\e} \right)\,d\mu_g\]
is a $SO(n)$ invariant unit-density Brakke flow on $S^{n+1}$ which converges to an equatorial sphere and to the same Clifford hypersurface, as $t \to \pm \infty$, respectively.

Moreover, there exists a $1$-parameter family $\{V_t(a)\}_{t \in \R}$ of Brakke flows, parametrized by $a \in S^1$, joining
the Clifford hypersurface $T_{1,n-1}$ to equatorial spheres $S^n$. Such flows are smooth for large $|t|$ and
depend equivariantly on $a$.
\end{thm}

\subsection*{Outline of the proof}

Using the rotation invariant property, the dimension reduction argument, and an inductive argument, we are able to generalize Cheng-Wei-Zeng's result \cite{MR4334413} to the stationary integral varifold case. The existence of stationary $SO(n)$ invariant solutions $u_{\e}^{-\infty}$ to \eqref{PAC} which have $T_{1,n-1}$ as their limit interface was shown by \cite{caju2019solutions, hiesmayr2020rigidity}. Using the classification of low area $SO(n)$ invariant stationary integral varifold, we can prove a rigidity result of low energy $SO(n)$ invariant critical points of $E_{\e}$.

Using the techniques from the previous work \cite{Chen2021mean}, we are able to construct $SO(n)$ invariant solutions $\{u_{\e}(\cdot, t)\}$ to \eqref{PAC}, connecting $u_{\e}^{-\infty}$ and a ground state solution (least energy unstable solutions of the Allen-Cahn equation, which are symmetric critical points with nodal sets exactly along equatorial spheres $S^n$ by \cite{caju2020ground}).

The convergence of $\{u_\e(\cdot, t)\}$ to an integral Brakke flow on $S^{n+1}$ can be derived from \cite{ilmanen1993convergence} and \cite{sato2008simple,tonegawa2003integrality}. This flow is \emph{cyclic mod 2}, in the sense of White \cite{white2009currents}. We then use the area bounds for the limit interfaces obtained from $u_\e^{\pm \infty}$, and the $SO(n)$ invariant property, to describe the forward and backward limit of this flow using the classification of low area $SO(n)$ invariant stationary integral varifolds in $S^{n+1}$. 

\subsection*{Organization}
In Section \ref{preliminary}, we state some results concerning the Allen-Cahn equation and give a brief description of the results from the previous paper. In Section \ref{symmetry and rigidity}, we study the symmetry and rigidity of solutions to \eqref{AC} with low energy. In Section \ref{sec:main}, we show the existence of $SO(n)$ invariant solutions to \eqref{PAC} connecting the Allen-Cahn approximation of the Clifford hypersurface and the equatorial sphere, and the corresponding Brakke flow. In the Appendix, we give proofs of several inequalities regarding the area of the minimal Clifford hypersurfaces. 

\subsection*{Acknowledgements}

We would like to express our gratitude to Andr\'e Neves for his support and numerous invaluable discussions and suggestions. Additionally, we extend our thanks to Pedro Gaspar, Yangyang Li, Daniel Stern, and Ao Sun for their insightful conversations. We are also appreciative of all the constructive comments provided by the referee.

\section{preliminary} \label{preliminary}

\subsection{The Allen-Cahn equation, induced varifolds and convergence} \label{ACminimal}

\

\begin{defi}
Let $(M^{n+1},g)$ be a Riemannian manifold. We define the \emph{Allen-Cahn energy} on $\Omega$ by:
\[E_{\epsilon}(u):= \int_{\Omega} \left(\frac{\epsilon}{2} |\nabla_{g} u|^2 + \frac{1}{\epsilon} W(u)\right) d\mu_{g}, \ \ u \in W^{1,2}(M),\]
where $d\mu_g$ is the volume measure with respect to $g$.
\end{defi}
Here $W(\cdot)$ is a double-well potential, the standard example is the function $W(t) = \frac{1}{4}(1 - t^2)^2$. Hereafter, we fix such a potential $W$.

One can check that $u$ is a critical point of $E_{\epsilon}$ on a closed manifold $(M^{n+1},g)$ if and only if $u$ (weakly) solves the \emph{elliptic Allen-Cahn equation}:
\begin{equation} \label{AC}
\epsilon^2 \Delta_g u -  W'(u)=0 \quad \text{on} \ M. \tag{AC}
\end{equation}

We write $\sigma = \int_{-1}^1 \sqrt{W(t)/2} dt$. This is the energy of the \emph{heteroclinic solution} $\mathbb{H}_\e(t)$ of \eqref{AC} on $\R$, namely, the unique bounded solution in $\R$ (modulo translation) such that $\mathbb{H}_\e(t) \to \pm 1$ when $t \to \pm \infty$. 

For the quadratic form given by the second variation of the energy $E_\e$ at $u$, we can define the \emph{Morse index} and the \emph{nullity} of a solution $u$ of \eqref{AC} (as a critical point of $E_\e$), denoted  $\ind_\e(u)$ and $\text{nul}_{\e}(u)$, as the number of negative eigenvalues and the dimension of the kernel of the linear operator 
    \[\mathcal{L}_{\e,u}(f) = \Delta f - \frac{W''(u)}{\e^2}f,\]
counted with multiplicity, respectively. 

We note that given $\e>0$ and a sufficiently regular function $u$ on $M$ (so that almost every level set is a regular hypersurface), we can consider the \emph{associated $n$-varifolds} $V_{\e,u}$ defined by
    \begin{equation} \label{def:varifold}
        V_{\e,u}(\phi) = \frac{1}{2}\int_{M\cap \{\nabla u \neq 0\}} \phi(x,T_x\{u=u(x)\})\cdot\left(\frac{\e|\nabla u(x)|^2}{2}+\frac{W(u(x))}{\e}\right)\,d\mu_g(x)
    \end{equation}
for any continuous function $\phi$ defined in the Grassmannian manifold $G_{n-1}(M)$, where $V_{\e,u}(\phi)$ denotes the integral of $\phi$ on $G_{n-1}(M)$ with respect to $V_{\e,u}$. We write $\mu_{\e,u}=\|V_{\e,u}\|$ for the associated Radon measure on $M$ (the \emph{weight measure} of $V_{\e,u}$). In the case where $\{u_j\}$ are solutions to \eqref{AC} or \eqref{PAC}, with $\e = \e_j \downarrow 0$, we will write ${V_{\e_j,t}=V_{\e_j,u_j(\cdot,t)}}$ and ${\mu_{\e_j,t} := \mu_{\e_j,u_j(\cdot,t)}}$.

We have the following convergence results for solutions to \eqref{AC} and \eqref{PAC}.

\begin{thm*}[\cite{hutchinson2000convergence,tonegawa2012stable,guaraco2018min}]
Let $(M^{n+1},g)$ be a closed Riemannian manifold. Let $\{u_j\}$ be a sequence of solutions of \eqref{AC} with $\e=\e_j \downarrow 0$. Suppose that $\sup_j E_{\e_j}(u_j) < \infty$. Then we can find a (not relabeled) subsequence of $u_j$ such that $V_{\e_j}$ converge to a stationary $n$ varifold $V$ on $M$ such that $\frac{1}{\sigma}V$ is integral. Moreover,
    \[\frac{1}{\sigma}\|V\|(M) = \lim_{j \to \infty} \frac{1}{\sigma} \|V_{\e_j}\|(M) = \lim_{j \to \infty}\frac{1}{2\sigma}E_{\e_j}(u_{j}),\]
and $u_{j}$ converges uniformly to $\pm 1$ in compact subsets of $M \setminus \supp \|V\|$.

Furthermore, if $n\geq 2$ and if $\sup_j \ind_{\e_j}(u_j)<\infty$, then $\supp\|V\|$ is a smooth, embedded, minimal hypersurface in $M$ away from a closed set of Hausdorff dimension $\leq (n-7)$.
\end{thm*}

The minimal surface $\supp\|V\|$ is often called a \emph{limit interface} obtained from $u_j$. Solutions with limit interface $\Gamma$ are called the \emph{Allen-Cahn approximation} of $\Gamma$.

We state below the main convergence result we will use in the present article, which follows from \cite{ilmanen1993convergence} and the work of Tonegawa \cite{tonegawa2003integrality} (see also \cite{sato2008simple} and \cite{takasao2016existence}):

\begin{thm*}
Let $(M^{n+1},g)$ be a closed Riemannian manifold. Let $\{u_j\}$ be a sequence of solutions to \eqref{PAC} on $M \times [t_0,\infty)$ with $\e=\e_j \downarrow 0$. Suppose that there exist constants $c_0,E_0>0$ such that
\begin{enumerate}
    \item[(a)] $\sup_{M\times [t_0,\infty)}|u_j| \leq c_0$, for all $j$,
    \item[(b)] $E_{\e_j}(u_j(\cdot,t)) \leq E_0$, for all $t\geq t_0$ and all $j$, and
    \item[(c)] $\int_{M\times(t_0,\infty)}\e_j|\partial_t u_j|^2\,d\mu_g \leq E_0$, for all $j$.
\end{enumerate}
Write $\mu_{\e_j,t} = \mu_{\e,u_j(\cdot, t)}$, for every $t\geq t_0$ and every $j$. Then, passing to a subsequence (not relabeled), there are Radon measures $\{\mu_t\}_{t\geq t_0}$ such that
\begin{enumerate}
    \item[(i)] $\mu_{\e_j,t} \to \mu_t$  as Radon measures on $M$, and
        \[\frac{1}{2\sigma}\lim_{j \to\infty} E_{\e_j}(u_j(\cdot,t)) = \frac{1}{\sigma}\lim_{j \to \infty}\|\mu_{\e_j,t}\|(M) = \frac{1}{\sigma}\|\mu_t\|(M),\]
    for every $t \in [t_0,\infty)$.
    \item[(ii)] For a.e. $t>t_0$, $\mu_t$ is $n$-rectifiable, and its density is $N(x)\sigma$, for $\mu_t$-a.e. $x \in M$, where $N(x)$ is a nonnegative integer.
    \item[(iii)] $\mu_t$ satisfies the mean curvature flow in the sense of Brakke, namely:
        \[\overline{D}_t \int_M \phi\,d\mu_t \leq \int_M (-\phi)\|H_t\|^2 + \langle\nabla \phi, H_t \rangle \, d\mu_t,\]
    for any $C^2$ function $\phi\geq 0$. Here $\overline{D}_t$ denotes the upper derivative, and $H_t$ is the generalized mean curvature vector of $\mu_t$.
\end{enumerate}
\end{thm*}

\subsection{Results and notations from previous paper}
\

We summarize the results from the previous paper \cite{Chen2021mean} joint with P. Gaspar. Some results and techniques will be adapted to be used in this article.

For a solution $u_{\e}^{-\infty}$ on $S^{n+1}$ to \eqref{AC} with Morse index $I$, by the result of Choi-Mantoulidis \cite{choi2022ancient}, where they show the existence of ancient gradient flows and uniqueness under integrability conditions, and the theory of parabolic PDEs, we show that there is a family of gradient flows of $E_{\e}$,
\begin{align*}
\mathscr{S}_{\e}: B_{\eta} \subset \R^I \to C^2(S^{n+1} \times \R)
\end{align*}
such that $\lim\limits_{t \to -\infty} \mathscr{S}_{\e}(\cdot, t) = u_{\e}^{-\infty}$. The curve $t \to \mathscr{S}_{\e}(a)(\cdot, t)$ is tangent at $u_{\e}^{-\infty}$ to eigenfunctions $\{\varphi_j\}_{j=1}^I$ of 
\begin{align*}
D^2E_{\e}(u_{\e}^{-\infty}) = \e \Delta - \frac{1}{\e} W''(u_{\e}^{-\infty})
\end{align*}
with coefficients $a = (a_1, \cdots, a_I)$.

When $I \geq 2$, for small $r < \eta$, using the fact that the first eigenfunction $\varphi_1 > 0$, the maximum principle for parabolic equations, and the Frankel-type property proved by Hiesmayr \cite{hiesmayr2020rigidity}, we prove that $\mathscr{S}_{\e}(\pm r, 0, \cdots, 0)$ converge to $\pm 1$ as $t \to +\infty$.

Let $\mathcal{S} = \{u \in C^2(M) \mid |u|\leq 1\}$. By Lemma 2.3 in \cite{GG} (and the continuous dependence of initial data, see e.g. Cazenave-Haraux \cite{CH}), there is a continuous map 
	\[\Phi:\mathcal{S} \times [0,\infty) \to W^{1,2}(M)\]
such that $\Phi(u,\cdot):M \times [0,\infty) \to \R$ is a solution of \eqref{PAC} defined for all $t \geq 0$ with $\Phi(u,0) = u$, and such that $\Phi(u,t) \in \mathcal{S}$, for all such $t$. 

By an analysis argument, we show that the sets
\begin{align*}
U_{\pm} = \{u \in \mathcal{S} \mid \|\Phi(t,u)- (\pm 1)\|_{W^{1,2}(M)} \to 0, \ \text{as} \ t\to +\infty\}
\end{align*}
are open. By connectedness, we show that along any curve connecting $(\pm r, 0, \cdots, 0) \in B_{\eta}$, there exists a point $a$ on it such that $\mathscr{S}_{\e}(a)$ does not converge to $\pm 1$ as $t \to +\infty$.

In the previous paper, we study solutions and flows on $S^3$. By analyzing the limit interfaces of the backward limit and the forward limit of $\mathscr{S}_{\e}(a)$ (on $S^3$) we mentioned above, which uses a parity argument and the classification of low area stationary integral varifold in $S^3$ based on the resolution of the Willmore conjecture by Marques-Neves \cite{marques2014min}, we show the existence of gradient flow of $E_{\e}$ connecting the Allen-Cahn approximations of a Clifford torus and an equatorial sphere.

Using the convergence results for solutions to \eqref{PAC} mentioned above, we show the limit of the gradient flows of $E_{\e}$ described above is a Brakke flow $\{V_t\}$. Using the energy estimate and the parity argument again, we show that the backward limit and the forward limit are a Clifford torus and an equatorial sphere, respectively.

By Choi-Mantoulidis \cite{choi2022ancient} rigidity of ancient gradient flows, we proved a symmetry argument for solutions to \eqref{PAC}. We described a relation between the isometries on $S^3$ and the isometries on the negative eigenspaces of the linearized Allen-Cahn operator at $u_{\e}^{-\infty}$. From this relation, we observe that if $a \in B_{\eta}$ is fixed by an isometry that preserves $u_{\e}^{-\infty}$, then so are $\mathscr{S}_{\e}(a)(\cdot, t)$, the limit Brakke flow $V_t$ and the limits
$V_{\pm \infty}$ (the limit of $V_t$ as $t \to \pm \infty$).

Using the symmetry argument and a continuity argument, we give a precise description of the backward limit and the forward limit, and therefore obtain a $2-$parameter family of Brakke flows generated by rotations on $S^3$.

In the previous paper, we study solutions and flows on $S^3$ that are invariant under the following $3$ isometries:
\begin{align*}
& \text{the reflection} \ (x_1,x_2,x_3,x_4) \to (x_2,x_1,x_3,x_4), \\
& \text{the reflection} \ (x_1,x_2,x_3,x_4) \to (x_1,x_2,x_4,x_3), \\
& \text{the evolution} \ (x_1,x_2,x_3,x_4) \to (x_3,x_4,x_1,x_2).
\end{align*}
We show the existence of a $2-$parameter family of weak mean curvature flows connecting a Clifford torus and equatorial spheres that are invariant under these isometries on $S^3$. In this article, we study solutions and flows on $S^{n+1}$ ($n \geq 2$) under a stronger symmetry assumption, that is invariant under $SO(n)$ (for instance, let $n = 2$, this condition means that the flows are invariant under all rotations acting on $x_3,x_4$ coordinates), and show the existence of a $1-$parameter family of weak mean curvature flows connecting $T_{1,n-1}$ and the equators $S^n$ in $S^{n+1}$ that are $SO(n)$ invariant.

\section{Symmetry and rigidity of solutions to the Allen-Cahn equation} \label{symmetry and rigidity}

\subsection{Symmetric stationary integral varifold}
\

Recall that $T_{1,n-1}$ is the Clifford hypersurface $S^1(\sqrt{\frac{1}{n}}) \times S^{n-1} (\sqrt{\frac{n-1}{n}})$ in $S^{n+1}$.

Cheng-Wei-Zeng \cite[Theorem 1.1]{MR4334413} proved that for a $n-$dimensional compact minimal rotational hypersurface in $S^{n+1}$, if it is neither a equatorial sphere $S^n$ nor a Clifford hypersurface $T_{1,n-1}$, then its area $> 2(1-\frac{1}{\pi}) \text{Area}(T_{1,n-1})$. 

\begin{rmk}
Otsuki \cite{otsuki1970minimal, otsuki1972integral} proved that there are no compact minimal embedded $SO(n)$ invariant hypersurfaces of $S^{n+1}$ other than Clifford hypersurfaces $T_{1,n-1}$ and round geodesic spheres. Li-Yau \cite{li1982new} proved that for the non-embedded case, the hypersurface has area of at least $2 \text{Area}(S^n)$. Combining these two results, we prove the Solomon-Yau conjecture for minimal $SO(n)$ invariant hypersurface. The above result gives us a stronger area bound (by using Lemma \ref{monotonicity and bound of ratio} below) since
\begin{align*}
2(1-\frac{1}{\pi}) \text{Area}(T_{1,n-1}) > 2(1-\frac{1}{\pi}) \cdot \sqrt{\frac{2\pi}{e}} \text{Area}(S^n) > 2 \text{Area}(S^n).
\end{align*}
\end{rmk}

We generalize their result to the stationary integral varifold case using an inductive argument and the dimension reduction argument. To do this, we establish the following monotonicity and boundedness result concerning the density ratio of $T_{1,n-1}$.

\begin{lem} \label{monotonicity and bound of ratio}
$\{\frac{\text{Area}(T_{1,n-1})}{\text{Area}(S^{n})}\}_{n \geq 2}$ is a strictly decreasing sequence, and $\sqrt{\frac{2\pi}{e}} < \frac{\text{Area}(T_{1,n-1})}{\text{Area}(S^n)} \leq \frac{\pi}{2}$ for all $n \geq 2$.
\end{lem}

The proof is computational, for detailed proof, see Appendix \ref{proof of density2}. 

\begin{lem}\label{minimals4}
Let $T$ be a $n-$dimensional stationary integral varifold in $S^{n+1}$ which is invariant under $SO(n)$. If $\text{Area}(T) \leq \text{Area}(T_{1,n-1})$ and its associated $\Z_2$ chain $[T]$ has $\partial [T] = 0$, then $T$ is a multiplicity one equator or a Clifford hypersurface $T_{1,n-1}$.
\end{lem}

\begin{proof}

We use an inductive argument. The case for $n = 2$ has been proved by Choi-Mantoulidis \cite[Lemma 5.8]{choi2022ancient} based on Marques–Neves's \cite{marques2014min} resolution of the Willmore conjecture and the dimension reduction argument. It is worth mentioning that in this case, we don't need the $SO(2)$ symmetry condition.

Now we prove the lemma for a $(n+1)-$dimensional stationary varifold $T$ in $S^{n+2}$ satisfying all conditions specified in the lemma, assuming that the lemma holds for the $n$-dimensional case.

If $T$ is smooth, the result follows by Cheng-Wei-Zeng \cite{MR4334413}.

We argue by contradiction to show that $T$ is smooth.

If $T$ were singular, then the $(n+2)-$dimensional stationary cone $C^{n+2}: = 0 \# T$ would be one where the origin is not an isolated singularity. If $x_0 \neq 0$ denotes a singular point of $C^{n+2}$, then by the monotonicity formula, the densities of $C^{n+2}$ at $x_0$ and the origin $0$ satisfy:
\begin{align*}
\Theta^{n+2}(C^{n+2},x_0) \leq \Theta^{n+2}(C^{n+2},0) = \frac{\text{Area}(T)}{\text{Area}(S^{n+1})} \leq \frac{\text{Area}(T_{1,n})}{\text{Area}(S^{n+1})}.
\end{align*}

By Lemma \ref{monotonicity and bound of ratio}, $\Theta^{n+2}(C^{n+2},x_0) < \frac{\text{Area}(T_{1,n-1})}{\text{Area}(S^{n})} \leq \frac{\pi}{2} < 2$.

Let $\overline{C}^{n+2}$ be a tangent cone to $C^{n+2}$ at $x_0$. If there exists a stationary $1-$dimensional cone $\overline{C}^1 \subset \R^2$ such that $\overline{C}^{n+2} \cong \overline{C}^1 \times \R^{n+1}$, then
\begin{align}\label{density}
\Theta(\overline{C}^1,0) = \Theta^{n+2}(C^{n+2},x_0) <2.
\end{align}

It is well known that all $1-$dimensional stationary cones are unions of $k \geq 2$ half-rays and have $\Theta^1(\overline{C}^1, 0 ) = \frac{k}{2}$. We have $k \leq 3$ by (\ref{density}). Moreover, $k$ is even because $\overline{C}^1$ is obtained by blow up of the $\Z_2$ cycle $T$. Therefore $k = 2$. This means $\overline{C}^1 \cong \R$ with multiplicity one and $\overline{C}^{n+2} \cong \R^{n+2}$ with multiplicity one. This violates the singular nature of $x_0 \in C^{n+2}$ by Allard's theorem \cite{simon1983lectures}.

Therefore $\overline{C}^{n+2}$ does not have the form $\overline{C}^1 \times \R^{n+1}$. By White \cite{white1997}, we know the set of points in $C^{n+2}$ at which no tangent cone has the form $\overline{C}^1 \times \R^{n+1}$ has Hausdorff dimension at most $n$. Since $T$ is $SO(n+1)$ invariant, $C^{n+2}$ is also $SO(n+1)$ invariant, by this symmetry property and the dimension comparison, we conclude that $x_0$ is the fixed point under $SO(n+1)$.

Now the tangent cone $\overline{C}^{n+2} \cong \overline{C}^{n+1} \times \R$ for some stationary $(n+1)-$dimensional cone $\overline{C}^{n+1} \subset \R^{n+2}$. Since $C^{n+2}$ and $x_0$ are $SO(n+1)$ invariant, then the tangent cone $\overline{C}^{n+2}$ is also $SO(n+1)$ invariant, thus $\overline{C}^{n+1}$ is $SO(n)$ invariant. Let $T^n \subset S^{n+1}$ be the link of $\overline{C}^{n+1}$, it is $SO(n)$ invariant, and
\begin{align*}
\text{Area}(T^n) = \text{Area}(S^n) \Theta^{n+1}(\overline{C}^{n+1}, 0) = \text{Area}(S^n) \Theta^{n+2}(C^{n+2}, x_0) < \text{Area}(T_{1,n-1}).
\end{align*}

Then by the $n-$dimensional case, we know that $T^n \simeq S^n$ with multiplicity one, so $\overline{C}^{n+1} \simeq \R^{n+1}$ with multiplicity one, and $\overline{C}^{n+2} \simeq \R^{n+2}$ with multiplicity one. This violates the singular nature of $x_0 \in C^{n+2}$ by Allard's theorem \cite{simon1983lectures}.

\end{proof}

\begin{rmk}
By Lemma \ref{monotonicity and bound of ratio}, we know that $\frac{\text{Area}(T_{1,n-1})}{\text{Area}(S^n)} > \sqrt{\frac{2\pi}{e}} > \frac{3}{2}$ for all $n \geq 2$, thus we could not rule out the possibility of $k = 3$ by improving the density estimate \eqref{density}. That is why we need to have the additional condition regarding the associated $\Z_2$ chain.
\end{rmk}

\subsection{Symmetry of the solution}
\

Caju-Gaspar \cite[Theorem $1.1$]{caju2019solutions} proved the existence of a solution $u_{\e}$ of \eqref{AC} whose nodal set converges to a minimal, separating hypersurface $\Gamma$, for sufficiently small $\e$, under the assumption that all Jacobi fields are generated by global isometries. They also proved that the Morse index and the nullity of $u_{\e}$ equal the Morse index and the nullity of $\Gamma$.

We here give a brief explanation of the index and the nullity estimate.  We can get a Jacobi field for the Allen-Cahn approximation of $\Gamma$ by pairing a Killing field with the gradient of this solution. These Jacobi fields for $u_{\e}$ produced by the Killing fields are linearly independent for sufficiently small $\e$, which implies that the nullity of $\Gamma \leq$ the nullity of the Allen-Cahn approximation of $\Gamma$. Then by the lower semicontinuity from \cite{gaspar2020second} and the upper semicontinuity in the multiplicity one case from \cite{chodosh2020minimal}, we claim that the Morse index and the nullity of $\Gamma$ and the Allen-Cahn approximation of $\Gamma$ are equal.

By Hsiang-Lawson \cite{hsiang1971minimal}, for the minimal hypersurface $T_{p,q}$ in $S^{n+1}$ ($n = p+q$), all Jacobi fields of $T_{p,q}$ are Killing-Jacobi fields. The nullity estimate above implies that the Jacobi fields for the Allen-Cahn approximation of $T_{p,q}$ produced by the Killing fields make up all the nullity of this solution for all sufficiently small $\e$.

Now we show the symmetry property of the solution $u_{\e}$. We include the proof from Hiesmayr \cite{hiesmayr2020rigidity} for completeness.

By using a change of variables argument as in \cite{hiesmayr2020rigidity}, we see that for any isometry $P$ of $S^{n+1}$ and any $C^1$ function $u$ defined on $S^{n+1}$, the pushforward $P_{\#}$ by the isometry $P$ satisfies
\begin{align*}
P_{\#}V_{\e,u} = V_{\e, u \circ P^{-1}}.
\end{align*}

Here we mention some common notations. We write $|\Sigma|$ for
the unit density varifold associated to $\Sigma$. Given a surface $\Sigma \subset S^{n+1}$, and a function $u$ on $S^{n+1}$, denote their stabilizer group by $\text{Stab} \Sigma$, $\text{Stab} u$, respectively. In other words, $\text{Stab} \Sigma = \{P \in SO(n+2) | P(\Sigma) = \Sigma\}, \text{Stab} u = \{P \in SO(n+2) | u \circ P = u\}$. These stabilizer groups are closed Lie subgroups of $SO(n+2)$, and we can therefore discuss the dimensions.

\begin{prop}[\cite{hiesmayr2020rigidity}] \label{symmetry}

For sufficiently small $\e$, the solution $u_{\e}$ of \eqref{AC} described above (which is the Allen-Cahn approximation of $T_{p,q}$) is $SO(p+1) \times SO(q+1)$ invariant, up to conjugation.

\end{prop}

\begin{proof}

It's clear that the Clifford hypersurface $T_{p,q}$ is invariant under $SO(p+1) \times SO(q+1)$, denoting its stabilizer group by $G$.

We argue by contradiction. If the Proposition does not hold, then there exists a sequence $\{\e_j\}$ such that $\e_j \downarrow 0$, the nodal sets of the solutions $u_j = u_{\e_j}$ of \eqref{AC} converge to $T_{p,q}$ as $j \to \infty$, and the stabilizer group of $u_j$ is not conjugated to $G$. 

Consider a sequence $\{P_j\}$ with $P_j \in \text{Stab} u_j$. Upon extracting a subsequence we may assume that it converges to some $P \in SO(n+2)$. Since $V_{\e_j, u_j} \to |T_{p,q}|$, we know that $P_{j\#} V_{\e_j,u_j} \to P_{\#} |T_{p,q}|$, while $P_{j\#} V_{\e_j,u_j} = V_{\e_j, u_j \circ P_j^{-1}} = V_{\e_j, u_j} \to |T_{p,q}|$. Hence $P_{\#}|T_{p,q}| = |T_{p,q}|$, and $P \in G$.

By a extraction argument, we find that given any $\tau > 0$, there is $J(\tau) \in \mathbb{N}$ so that $\text{Stab} u_j \subset (G)_{\tau}$ when $j \geq J(\tau)$. Here the Lie group $SO(n+2)$ is endowed with a bi-invariant metric, and $(G)_{\tau}$ is the open tubular neighborhood of $G$ of size $\tau >0$. By \cite{MR6545}, $\text{Stab} u_j$ is conjugate to a subgroup of $G$. By a nullity estimate argument above, we know that $\dim \text{Stab} u_j = \dim \text{Stab} T_{p,q}$, thus $\text{Stab} u_j$ is conjugate to $G$, we get a contradiction.

\end{proof}

\begin{rmk}

For equatorial sphere $S^n \subset S^{n+1}$, it also satisfies the assumption that all Jacobi fields are Killing-Jacobi fields (see \cite{simons1968minimal}). By the same argument, we know that the Allen-Cahn approximation of the equatorial sphere is $SO(n+1)$ invariant, for sufficiently small $\e$.

\end{rmk}

Recall from \cite{caju2020ground}, ground state solutions are unstable solutions of least energy. In the Allen-Cahn setting, the ground state solution on $S^{n+1}$ is unique up to rigid motions, and its nodal set is exactly along an equator $S^n$ for sufficiently small $\e$. 

One may combine \cite{hardt1989nodal} with \cite{brezis1986remarks} to obtain the following general result about the nodal sets: Let $(M,g)$ be closed, let $\e > 0$ and $u_{\e}^1, u_{\e}^2$ be two solutions of \eqref{AC} on $M$. If the nodal sets of $u_{\e}^1, u_{\e}^2$ are same, then $u_{\e}^1 = \pm u_{\e}^2$.

\begin{coro} \label{rigidity of ground state}
For sufficiently small $\e$, if $u_{\e}$ is a Allen-Cahn approximation of the equatorial sphere $S = S^n \subset S^{n+1}$, then $u_{\e}$ is a ground state solution.
\end{coro}

\begin{proof}

We know that $u_{\e}$ is $SO(n+1)$ invariant. By \cite[Lemma 7.2]{caju2020ground}, which used a rotation argument and the maximum principle, we know that the nodal set of $u_{\e}$ is exactly the equator which is invariant under the stabilizer group of $u_{\e}$. We claim that $u_{\e}$ is a ground state solution.

\end{proof}

\subsection{Solutions on \texorpdfstring{$S^{n+1}$}{S4}}
\

From now on we focus on the Clifford hypersurface $T_{1,n-1}$, that is, up to isometry, the minimal hypersurface
\begin{equation} \label{T4}
T_c = \{(x_1,x_2,\cdots, x_{n+2}) \in \R^{n+2}: x_1^2 + x_2^2 = \frac{1}{n}, x_3^2 + \cdots + x_{n+2}^2 = \frac{n-1}{n}\} \subset S^{n+1}.
\end{equation}

We study the specific Clifford hypersurface $T_c$ (as defined in equation \eqref{T4}) and its Allen-Cahn approximation, as all other Clifford hypersurfaces $T_{1,n-1}$ (and their Allen-Cahn approximation) in $S^{n+1}$ are just different by a rotation. Denote the Allen-Cahn approximation of $T_c$ as $u_{\e}^{-\infty}$, and we know that $E_{\e}(u_{\e}^{-\infty}) \to (2\sigma) \cdot \text{Area}(T_c)$ as $\e \downarrow 0$. 

Let the group $G_1$ denote all rotations acting on the $x_1,x_2$ coordinates, and the group $G_2$ denote all rotations acting on the $x_3,\cdots,x_{n+2}$ coordinates. It is clear that $T_c$ is $G_1 \times G_2$ invariant.

\begin{lem} \label{existence of symmetric solution}
For sufficiently small $\e$, there exists an Allen-Cahn approximation $u_{\e}^{-\infty}$ of $T_c$ that is invariant under $G_1 \times G_2$.
\end{lem}

\begin{proof}

Let $u_{\e}$ be the Allen-Cahn approximation of the Clifford hypersurface $T_c$, by Proposition \ref{symmetry}, we know that $u_{\e}$ is $SO(2) \times SO(n)$ invariant, up to conjugation.

Therefore there exists an isometry $P_1 \in SO(n+2)$ such that $u_{\e} \circ P_1$ is invariant under a group $G = SO(2) \times SO(n)$. It is clear that $G$ and $G_1 \times G_2$ are conjugate, so there exists $P_2 \in SO(n+2)$ such that $P_2^{-1}G P_2 = G_1 \times G_2$.

Then $u_{\e}^{-\infty} = u_{\e} \circ P_1 \circ P_2$ is a solution to \eqref{AC} that is invariant under $G_1 \times G_2$. The remaining step is to show that $u_{\e}^{-\infty}$ is an Allen-Cahn approximation of $T_c$.

We know that $E_{\e}(u_{\e}^{-\infty}) = E_{\e}(u_{\e}) \to 2\sigma \cdot \text{Area}(T_c)$ as $\e \downarrow 0$, so the energy of $u_{\e}^{-\infty}$ is bounded for small $\e$. Upon extracting a subsequence, the limit interface of $u_{\e}^{-\infty}$ is a Clifford hypersurface $T_{1,n-1}$. Since the limit interface of $u_{\e}^{-\infty}$ needs to be invariant under $G_1 \times G_2$, and the only $G_1 \times G_2$ invariant Clifford hypersurface is $T_c$, we claim that the limit interface of $u_{\e}^{-\infty}$ is the Clifford hypersurface $T_c$.

\end{proof}

Next, we show that the only $SO(n)$ invariant solutions of \eqref{AC} with energy level under $u_{\e}^{-\infty}$ are ground states.

\begin{lem} \label{equatorial}
For sufficiently small $\e$, the solution $u_\e^{-\infty}$ has Morse index $n+3$, and the only nonconstant $SO(n)$ invariant solutions of \eqref{AC} with energy $<E_\e(u_\e^{-\infty})$ are ground states.
\end{lem}

\begin{proof}

Since the Clifford hypersurface $T_c$ has Morse index $n+3$, by the index estimate above, we know that the Morse index of $u_{\e}^{-\infty}$ is $n+3$ for sufficiently small $\e$.

We argue by contradiction. If the Lemma does not hold, then there exists a sequence $\e_j>0$ such that $\e_j \downarrow 0$, and a sequence $u_j$ of nonconstant $SO(n)$ invariant solutions to \eqref{AC} with $\e=\e_j$ which are not ground states and have energy $<E_{\e_j}(u_{\e_j}^{-\infty})$.

Since these solutions have uniformly bounded energy, by passing to a subsequence, we may assume that the varifolds $\frac{1}{\sigma}V_{\e_j,u_j}$ converge to a stationary integral varifold $\frac{1}{\sigma}V$ in $S^{n+1}$ with area $\leq \text{Area}(T_c)$. Since all varifolds $\frac{1}{\sigma}V_{\e_j,u_j}$ are $SO(n)$ invariant, therefore $\frac{1}{\sigma}V$ is also $SO(n)$ invariant.

We claim that $\frac{1}{\sigma}V$ has density $\mathcal{H}^2$-a.e. equal to $1$ on its support. In fact, $\Theta(\frac{1}{\sigma}V,x) \in \mathbb{Z}_{+}$ for $\mathcal{H}^2$-almost every such $x$. From $\frac{1}{\sigma}\|V\|(S^{n+1})\leq \text{Area}(T_c) \leq \frac{\pi}{2} \text{Area}(S^n) < 2 \text{Area}(S^n)$ and the density estimate in \cite[Lemma A.2]{marques2014min}, we see that the density of $\frac{1}{\sigma}V$ is everywhere strictly less than $2$, proving the claim.

By the remarks about energy loss in \cite{hutchinson2000convergence}, we see that $\frac{1}{\sigma}V$ is the boundary of a region. More precisely, the varifold $\frac{1}{\sigma}V$ agrees with the multiplicity one varifold induced by the reduced boundary of $\{u=1\}$, where $u$ is the function of bounded variation on $S^{n+1}$ given by the a.e. limit of $u_j$. Consequently, the boundary of the $\mathbb{Z}_2$ chain associated to $\frac{1}{\sigma}V$ (in the sense of White \cite{white2009currents}) vanishes. By Lemma \ref{minimals4}, it follows that $\frac{1}{\sigma}V$ is either a multiplicity one equatorial sphere or a Clifford hypersurface $T_{1,n-1}$.

If the limit interface $\supp||V||$ is an equator, then by Corollary \ref{rigidity of ground state}, $u_j$ is a ground state, we get a contradiction.

If the limit interface $\supp||V||$ is a Clifford hypersurface $T_{1,n-1}$, then by Hiesmayr's rigidity result, $u_j$ equals to $u_{\e_j}^{-\infty}$ up to isometry for large $j$. This contradicts the energy bounds $E_{\e_j}(u_j) < E_{\e_j}(u_{\e_j}^{-\infty})$.
\end{proof}

\section{main results} \label{sec:main}

\subsection{Gradient flow of the energy functional}
\

For the critical point $u_{\e}^{-\infty}$ which is the Allen-Cahn approximation of $T_c$, we use the results of \cite{Chen2021mean} to construct $G_2$ invariant gradient flow of $E_{\e}$ with $u_{\e}^{-\infty}$ as its backward limit.

For sufficiently small $\e$, $u_{\e}^{-\infty}$ has Morse index $n+3$, we denote by $\{\varphi_i\}_{i=1}^{n+3}$ an $L^2$-orthonormal basis for the eigenspaces of the linearized Allen-Cahn operator at $u_\e^{-\infty}$ corresponding to negative eigenvalues, where we assume $\varphi_1>0$.

We also recall the solution map $\mathscr{S}_\e=\mathscr{S}\colon B_\eta(0)\subset \R^{n+3} \to C^{2,\alpha}(S^{n+1} \times \R)$, defined for some $\eta=\eta_\e>0$. Using the same argument as in \cite{Chen2021mean} Section $5$, we can choose the eigenfunctions $\varphi_i$ in a way that $G_2$ (rotations on $S^{n+1}$ acting on $x_3, \cdots, x_{n+2}$ coordinates) acts on $\varphi_4, \cdots, \varphi_{n+3}$ by rotations and fixes $\varphi_2,\varphi_3$ pointwise. Note that $\varphi_1$ is fixed since $\varphi_1 > 0$ and the rotations $G_2$ fix $T_c$. 

For each small $\e > 0$, we fix $r = r(\e) \in (0,\eta)$ depending continuously on $\e$.
As shown in the proof of \cite[Proposition 1]{Chen2021mean}, $\mathscr{S}(\pm r, 0, \dots, 0)$ converge to $\pm 1$ as $t \to +\infty$, and for any path connecting $(\pm r, 0, \dots, 0)$ in $\partial B_r(0)$, there is a point $a$ on this path such that $\mathscr{S}(a)$ do not converge to the constant critical points $\pm 1$ of $E_{\e}$. 

Consider the path $l: [-1, 1] \to \partial B_r(0), l(t) = (rt, r\sqrt{1 - t^2}, 0, \dots, 0)$, connecting two points $(\pm r, 0, \dots, 0)$. As mentioned above, we know that there exists $t \in (-1,1)$ such that $\mathscr{S}(l(t))$ does not converge to the constant critical points $\pm 1$. Denote the point $l(t)$ by $a_{\e}$. Since any rotation in $G_2$ fixes $\varphi_1, \varphi_2, \varphi_3$, so $\mathscr{S}(a_\e)$ is $G_2$ invariant.

Combine with Lemma \ref{equatorial}, we can prove the following Proposition:

\begin{prop} \label{AC flow}

For sufficiently small $\e > 0$, there are eternal solutions $\{u_{\e}\}$ of \eqref{PAC} on $S^{n+1}$ such that 
\begin{align*}
u_\e^{-\infty} = \lim_{t \to -\infty} u_\e(\cdot, t) \quad \text{and} \quad u_\e^{+\infty} = \lim_{t \to +\infty} u_\e(\cdot ,t),
\end{align*}
where $u_{\e}^{-\infty}$ is the Allen-Cahn approximation of $T_c$ described above, and $u_{\e}^{+\infty}$ is a ground state solution. Moreover, the solutions $\{u_{\e}\}$ are $G_2$ invariant.

\end{prop}

\begin{proof}
It remains to prove the full convergence of $u_\e(\cdot,t)$, as $t \to +\infty$. By \cite{caju2020ground}, least energy unstable critical points of the energy functional $E_\e$ are unique up to ambient isometries. In particular, $E_\e$ is a \emph{Morse-Bott} functional at this critical level. Thus, the convergence of $u_\e(\cdot, t)$ to $u_\e^{+\infty}$ in $W^{1,2}$ is a consequence of the Łojasiewicz-Simon gradient inequality for such functionals, see e.g. \cite{feehan2020lojasiewicz}.
\end{proof}

\begin{rmk}

In particular, any limit interface obtained from the limits $u_{\e}^{+\infty}$ is a multiplicity one equatorial sphere, and it holds $E_{\e}(u_{\e}^{+\infty}) \to (2\sigma) \cdot \text{Area}(S^n)$.

\end{rmk}

\subsection{Limit flow}
\

Now we analyze the limit of the gradient flow given by Proposition \ref{AC flow} as $\e \downarrow 0$. Using the similar argument as in Section $4.2$ in the author's previous joint work with P. Gaspar \cite{Chen2021mean},  we can show that the gradient flow satisfies the necessary conditions to take the limit as $\e \downarrow 0$ and obtain a codimension one Brakke flow on the sphere $S^{n+1}$.

Since $E_{\epsilon} (u_{\epsilon}^{-\infty}) \to 2\sigma \text{Area}(T_{1,n-1})$, and $E_{\epsilon}(u_\epsilon^{+\infty}) \to 2\sigma \text{Area}(S^n)$. Thus given a small $\delta > 0$, for sufficiently small $\epsilon > 0$ (depending on $\delta$), we have 
    \[2\sigma \text{Area}(T_{1,n-1}) - \delta \leq E_{\epsilon} (u_{\epsilon}^{-\infty}) \leq 2\sigma \text{Area}(T_{1,n-1}) + \delta \quad \text{and} \quad E_{\epsilon} (u_{\epsilon}^{+\infty}) \leq 2\sigma \text{Area}(S^n) + \delta.\]

Recall that the energy $E_{\epsilon}(u_{\epsilon}(\cdot,t))$ is a continuous strictly decreasing function of $t$. By picking a sufficiently small $\delta>0$ and by noting that this solution joins $u_{\epsilon}^{-\infty}$ to $u_{\epsilon}^{+\infty}$, we see that there exists $t(\epsilon) \in \R$ such that $E_{\epsilon}\left(\, u_{\epsilon}(\cdot,t(\epsilon))\,\right) = 2\sigma (1.2 \text{Area}(S^n))$ (as $\text{Area}(S^n) < 1.2 \text{Area}(S^n) < \text{Area}(T_{1,n-1})$ by Lemma \ref{monotonicity and bound of ratio}). By translating the gradient flow $u_\epsilon(\cdot,t)$ to $u_{\epsilon}(\ \cdot \ , t+t(\epsilon))$, we can assume that $E_{\epsilon}\left(\,u_{\epsilon}(\cdot,0)\,\right) = 2\sigma (1.2 \text{Area}(S^n))$ for all small $\epsilon$.

By the convergence result for solutions of \eqref{PAC} of Ilmanen \cite{ilmanen1993convergence} and Tonegawa \cite{tonegawa2003integrality} (see also Sato \cite{sato2008simple}), after passing to a subsequence (not relabeled) with $\e\downarrow 0$, the varifolds $V_{\e,t}$ associated to $u_{\epsilon}(\cdot,t)$ converge, for every $t \in \R$, to a $n$-varifold $V_t$, and the underlying Radon measures $\frac{1}{\sigma}\mu_{\e,t}=\frac{1}{\sigma}\|V_{\e,t}\|$ converge to a Radon measure       \begin{equation}\Sigma_t:=\frac{1}{\sigma}\mu_t=\frac{1}{\sigma}\|V_t\| \end{equation}
which satisfies the mean curvature flow equation in the sense of Brakke. Moreover,
     \[\frac{1}{2\sigma} E_{\epsilon}\left(\,u_{\epsilon}(\ \cdot \ ,t)\,\right) \to \|\Sigma_t\| (S^{n+1}), \quad \text{as} \quad \epsilon \downarrow 0,\]
and, for almost every $t \in \R$, the varifold $\frac{1}{\sigma}V_t$ is an integral varifold.
   
More precisely, we apply the convergence result to $u_\e$ on $S^{n+1} \times [-m,+\infty)$ for each $m \in \mathbb{N}$ to obtain the (subsequential) convergence of $\frac{1}{\sigma}V_{\e,t}$ for all $t \geq -m$. By picking a diagonal subsequence, we get the convergence described above.\medskip

Since $\frac{1}{2\sigma} E_{\epsilon}\left(\,u_{\epsilon}(\ \cdot \ ,0)\,\right) = 1.2 \text{Area}(S^n)$ for all small $\epsilon$, we see that 
    \[\|\Sigma_{-t}\|(S^{n+1}) = \lim_{\e \downarrow 0} \frac{1}{2\sigma}E_{\e}(u_\e(\cdot,-t)) \geq \lim_{\e \downarrow 0} \frac{1}{2\sigma}E_{\e}(u_\e(\cdot,0)) = 1.2 \text{Area}(S^n)\]
and, similarly, $\|\Sigma_t\|(S^{n+1}) \leq 1.2 \text{Area}(S^n)$, for every $t \geq 0$. Note also that $\|\Sigma_t\|(S^{n+1}) \leq \text{Area}(T_{1,n-1})$ for all $t \in \R$. In fact, since $\frac{1}{2\sigma}E_{\epsilon} (u^{-\infty}_{\epsilon}) \to \text{Area}(T_{1,n-1})$ as $\epsilon \downarrow 0$, for each $\delta' > 0$, there exists a $\epsilon' > 0$ small, such that 
    \[\frac{1}{2\sigma}E_{\epsilon} (u^{-\infty}_{\epsilon}) < \text{Area}(T_{1,n-1}) + \delta', \quad \text{for all} \  \epsilon \in (0,\epsilon').\]
By noting that $E_{\epsilon}\left(\, u_{\epsilon}(\ \cdot 
\ ,t)\,\right)\leq E_\e(u_\e^{-\infty})$ for all $t$, we obtain $\|\Sigma_t\|(S^{n+1}) \leq \text{Area}(T_{1,n-1}) + {\delta'}$, for every $t$. Since $\delta'$ is arbitrary, this implies that $\|\Sigma_t\|(S^{n+1})\leq \text{Area}(T_{1,n-1})$.

By abuse of notation, we will identify $\Sigma_t$ with its support, which is, for almost every $t \in \R$, a $n$-dimensional rectifiable set. We want to show that $\Sigma_t$ (and the associated varifolds $\frac{1}{\sigma}V_t$) converge to a multiplicity one Clifford hypersurface, as $t \to -\infty$  along subsequences, and to a multiplicity one equatorial sphere, as $t \to +\infty$, also along subsequences.

We now work on the symmetry property of $\Sigma_t$. 
Since $\mathscr{S}(a_\e)$ is $G_2$ invariant, then the varifolds $V_{\e, \mathscr{S}(a_\e)(\cdot, t)}$ is also $G_2$ invariant, so are the varifolds after time translation, the limit Brakke flow $\Sigma_t$ and the limits of $\Sigma_t$ as $t \to \pm \infty$.

We will need further information about the parity of the multiplicity of the limit varifold $V_t$, as described in the following lemma. It intuitively says that the interfaces fold an even number of times near a point in $\supp\|V_t\|$ if, and only if, $u_\e(\cdot, t)$ converges to the same value on the two sides of this surface. This was proved by Hutchinson-Tonegawa \cite{hutchinson2000convergence} in the elliptic case; see also Takasao-Tonegawa \cite{takasao2016existence} in the parabolic case, for an equation with a transport term in Euclidean domains or in a torus.

\begin{lem} \label{odd-even-density}
For almost every $t\in \R$, the density of the varifold $\Sigma_t$ satisfies 
\[\Theta(\Sigma_t,x)=\left\{
\begin{aligned}
\rm{odd}&  \ \ \  \mathcal{H}^{2}\text{-a.e.} \  x \in M_t, \\
\rm{even}&  \ \ \  \mathcal{H}^{2}\text{-a.e.}\ x \in \mathrm{supp}\|\Sigma_t\| \setminus M_t,
\end{aligned}
\right.\]
where $M_{t}$ is the reduced boundary of $\{u_0(\cdot,t) = 1\}$, and $u_0(\cdot,t)$ is the bounded variation function given by the weak-$*$ limit of $u_\e(\cdot,t)$, as functions of bounded variation.
\end{lem}

Finally, we can describe the limits of the Brakke flow $\Sigma_t$ in $S^{n+1}$.

\begin{thm} \label{graphical}
As $t \to -\infty$, the varifold $\frac{1}{\sigma}V_t$ converges to a multiplicity one Clifford hypersurface $T_{1,n-1}$ in $S^{n+1}$, and its support converges graphically to this torus.
As $t \to +\infty$, the varifold $\frac{1}{\sigma}V_t$ subconverges to a multiplicity one equatorial sphere.
\end{thm}

\begin{proof}
Consider any sequence $\theta_i \uparrow \infty$, and the sequence of translated Brakke flows $\{\Sigma_t^{(i)} := \Sigma_{t+ \theta_i}\}_{t \geq 0}$. By Brakke’s compactness theorem and the uniform boundedness of areas, $(\Sigma_t^{(i)})_{t\geq 0}$ converges subsequentially to an integral Brakke flow with constant area. Therefore, this Brakke flow is supported on a stationary integral varifold $\frac{1}{\sigma} V_{+\infty}$. Similarly, $\Sigma_t$ subconverges, as $t \to -\infty$, to a stationary integral varifold $\frac{1}{\sigma} V_{-\infty}$.

By Lemma \ref{odd-even-density}, the associated $\mathbb{Z}_2$ chain of $\frac{1}{\sigma}V_t$ is $M_0^{t}$ is the reduced boundary of $\{u_0^{t} = 1 \}$, thus it has vanishing boundary. By White \cite{white2009currents} Theorem $4.2$, we know that the associated $\mathbb{Z}_2$ chain of any subsequential limit varifold $\frac{1}{\sigma} V_{+\infty}$ or $\frac{1}{\sigma} V_{-\infty}$ has zero boundary.

We have shown the area estimate $\|\Sigma_t\|(S^{n+1}) \leq \text{Area}(T_{1,n-1})$ and $\Sigma_t$ is $G_2$ invariant. By Lemma \ref{minimals4}, it follows that any such limit $\frac{1}{\sigma} V_{-\infty}$ or $\frac{1}{\sigma} V_{+\infty}$ is either a multiplicity one equatorial sphere or a multiplicity one  Clifford hypersurface $T_{1,n-1}$. On the other hand, we have the inequality $\|\Sigma_{-t}\|(S^{n+1}) \geq 1.2 \text{Area}(S^n) \geq \|\Sigma_t\|(S^{n+1})$ for every $t \geq 0$. Thus $\frac{1}{\sigma} \| V_{-\infty}\|(S^{n+1}) \geq 1.2 \text{Area}(S^n) \geq \frac{1}{\sigma} \|V_{+\infty}\|(S^{n+1})$, and we conclude that any subsequential limit $\frac{1}{\sigma} V_{-\infty}$ is a Clifford hypersurface $T_{1,n-1}$, and any subsequential limit $\frac{1}{\sigma} V_{+\infty}$ is a multiplicity one equatorial sphere.

The convergence as $t \to -\infty$ and the graphical convergence of $\Sigma_t$ to the minimal torus, follows from \cite{choi2022ancient}, by means of Brakke's local regularity for the mean curvature flow \cite{Brakke}, as the characterization of the backward limit allows us to obtain the smoothness of $\Sigma_t$ for sufficiently negative time.
\end{proof} 

\begin{rmk}

Since the backward limit and the forward limit of the Brakke flow $\Sigma_t$ are $G_2$ invariant, i.e., invariant under any rotations on the $x_3, \cdots, x_{n+2}$ coordinates. It's not hard to see that the backward limit needs to be the Clifford hypersurface $T_c$, which is also fixed by the rotations on the $x_1,x_2$ coordinates. Since all equatorial spheres that are invariant under $G_2$ are not fixed under rotations on the $x_1,x_2$ coordinates, therefore by rotations on the $x_1,x_2$ coordinates, we get a one-parameter family of Brakke flows on $S^{n+1}$, which converges to an equatorial sphere and to the Clifford hypersurface $T_c$, as $t \to \pm \infty$, respectively. This concludes the last part of Theorem \ref{main}.

\end{rmk}

\appendix

\section{Proof of Lemma \ref{monotonicity and bound of ratio}} \label{proof of density2}

We require the following two lemmas to prove Lemma \ref{monotonicity and bound of ratio}. The proof for the second lemma is omitted as it can be proved similarly to the first one.

\begin{lem} \label{h}
Let $h(x) = \sqrt{\frac{(x+1)^{x+3}x^{x-2}}{(x+2)^{x+2}(x-1)^{x-1}}}$ be the function defined on $[2,+\infty)$, then $h(x)$ is an increasing function with limit $1$ at $+\infty$. Therefore $h(x) < 1$.
\end{lem}

\begin{proof}

One can see that $h(x) > 0$. By computation, we get
\begin{align*}
& \frac{d}{dx} \ln(h(x)) = \frac{1}{2}(\ln(x+1) + \ln(x) - \ln(x+2) - \ln(x-1) - \frac{2}{x(x+1)}), \\
& \frac{d^2}{dx^2} \ln(h(x)) = - \frac{1}{(x^2+x-2)(x^2+x)^2} (4x+2) < 0.
\end{align*}

Then we know that $\frac{d}{dx} \ln(h(x))$ is a decreasing function, and it's easy to check that its limit as $x \to +\infty$ is $0$, thus $\frac{d}{dx} \ln(h(x)) > 0$. Therefore $h(x)$ is an increasing function.

For the limit of $h(x)$ as $x \to +\infty$, we have the following computation:

\begin{align*}
\lim\limits_{x \to +\infty} h(x) = \lim\limits_{x \to +\infty} \sqrt{(\frac{x^2 + x}{x^2+x-2})^{x-2} \cdot \frac{(x+1)^5}{(x+2)^4(x-1)}} = \lim\limits_{x \to +\infty} \sqrt{(1 + \frac{2}{x^2+x-2})^{x-2}} = 1.
\end{align*}

\end{proof}

\begin{lem} \label{m}
Let Let $m(x) = \sqrt{\frac{x^{2x-2}(x+1)^4 (x+3)^{x+3}}{(x-1)^{x-1}(x + 2)^{2x + 6}}}$ be the function defined on $[2,+\infty)$, then $m(x)$ is an increasing function with limit $1$ at $+\infty$. Therefore $m(x) < 1$.
\end{lem}

\begin{proof}[Proof of Lemma \ref{monotonicity and bound of ratio}]

Let \begin{align*}
d(n) = \frac{\text{Area}(T_{1,n-1})}{\text{Area}(S^n)} = \frac{\text{Area}(S^1(\sqrt{\frac{1}{n}}) \times S^{n-1} (\sqrt{\frac{n-1}{n}}))}{\text{Area}(S^n)}   
\end{align*}

We have the following equality regarding the surface area of the $n-$sphere and the $(n+2)-$sphere:
\begin{equation} \label{+2 ratio}
\frac{\text{Area}(S^{n+2})}{\text{Area}(S^n)} = \frac{2\pi}{n+1}.
\end{equation}

Then we know that:
\begin{equation}\label{dn ratio}
\begin{split}
\frac{d(n+2)}{d(n)} &= \frac{\text{Area}(S^1(\sqrt{\frac{1}{n+2}}))}{\text{Area}(S^1(\sqrt{\frac{1}{n}}))} \cdot \frac{\text{Area}( S^{n+1} (\sqrt{\frac{n+1}{n+2}}))}{\text{Area}( S^{n-1} (\sqrt{\frac{n-1}{n}}))} \cdot \frac{\text{Area}(S^n)}{\text{Area}(S^{n+2})} \\
&= \sqrt{\frac{n}{n+2}} \cdot \sqrt{\frac{(n+1)^{n+1}n^{n-1}}{(n+2)^{n+1}(n-1)^{n-1}}} \cdot \frac{2\pi}{n} \cdot \frac{n+1}{2\pi} \\
&= \sqrt{\frac{(n+1)^{n+3}n^{n-2}}{(n+2)^{n+2}(n-1)^{n-1}}} \\
&= h(n) < 1,
\end{split}
\end{equation}
where we use Lemma \ref{h} at the last step. By computation, $d(2) = \frac{\pi}{2}, d(3) = \frac{8}{3\sqrt{3}} < \frac{\pi}{2}$. Therefore $d(n) \leq \frac{\pi}{2}$ for all $n \geq 2$.

Using the equation \eqref{dn ratio} and Lemma \ref{m}, we know
\begin{equation} \label{increasing ratio}
\frac{d(n+1)}{d(n)} = \frac{d(n+3)}{d(n+2)} \sqrt{\frac{n^{2n-2}(n+1)^4 (n+3)^{n+3}}{(n-1)^{n-1}(n + 2)^{2n + 6}}} = \frac{d(n+3)}{d(n+2)} m(n) < \frac{d(n+3)}{d(n+2)},
\end{equation}
for all $n \geq 2$.

We know $d(n) = 2\pi \sqrt{\frac{(n-1)^{n-1}}{n^n}} \frac{\text{Area}(S^{n-1})}{\text{Area}(S^n)}$. Recall the surface area formula:
\begin{equation} \label{surface area formula}
\text{Area}(S^n) = \left\{
\begin{aligned}
(n+1)\pi^{\frac{n+1}{2}}/(\frac{n+1}{2})! &    & \rm{n \ is \  odd}, \\
(n+1) \pi^{\frac{n}{2}}2^{\frac{n+2}{2}}/(n+1)!!  &     & \rm{n \ is \ even}.
\end{aligned}
\right.
\end{equation}

We can compute that:
$$ \frac{d(n+1)}{d(n)} = \left\{
\begin{aligned}
\sqrt{\frac{n^{2n}}{(n+1)^{n+1}(n-1)^{n-1}}} \cdot \frac{\pi (n+1)^2}{2n} (\frac{n \cdot (n-2) \cdots 1}{(n+1) \cdot (n-1) \cdots 2})^2,&    & \rm{n \ is \  odd}, \\
\sqrt{\frac{n^{2n}}{(n+1)^{n+1}(n-1)^{n-1}}} \frac{2 (n+1)^2}{\pi n} (\frac{n \cdot (n-2) \cdots 2}{(n+1) \cdot (n-1) \cdots 1})^2,&     & \rm{n \ is \ even}.
\end{aligned}
\right.
$$

By Stirling's approximation, $n! \sim \sqrt{2\pi n} (\frac{n}{e})^n$, one can verify that $\frac{d(n+1)}{d(n)}$ converges to $1$ as $n \to +\infty$. When combined with the monotonicity described in \eqref{increasing ratio}, we can deduce that $\frac{d(n+1)}{d(n)} < 1$ for all $n \geq 2$. Hence $\{d(n)\}_{n \geq 2}$ is a strictly decreasing sequence.

We are able to bound $\frac{\text{Area}(T_{1,n-1})}{\text{Area}(S^n)}$ from its limit. By computation:
\begin{align*}
\lim\limits_{n \to + \infty} d(n) d(n+1) &= \lim\limits_{n \to + \infty}  (2\pi)^2 \cdot \sqrt{\frac{(n-1)^{n-1}}{(n+1)^{n+1}}} \cdot \frac{\text{Area}(S^{n-1})}{\text{Area}(S^{n+1})}\\
&= \lim\limits_{n \to + \infty}  (2\pi)^2 \cdot \sqrt{(1 - \frac{2}{n+1})^{n-1}} \cdot \frac{1}{n+1} \cdot \frac{n}{2\pi} \\
&= \frac{2\pi}{e}.
\end{align*}

Hence $\lim\limits_{n \to +\infty} d(n) = \sqrt{\frac{2\pi}{e}}$. We conclude the lower bound from the monotonicity.

\end{proof}

\section{Clifford hypersurface with lowest area}

Let $T_{min}^n$ denote the Clifford hypersurface in $S^{n+1}$ with lowest area among all Clifford hypersurfaces $T_{p,q}$ ($p+q =n$) in $S^{n+1}$. We give the explicit expression of $T_{min}^n$:

\begin{prop} \label{min}
When $n$ is even, $T_{min}^n = T_{\frac{n}{2}, \frac{n}{2}} = S^{\frac{n}{2}}(\sqrt{\frac{1}{2}}) \times S^{\frac{n}{2}}(\sqrt{\frac{1}{2}})$.

When $n$ is odd, $T_{min}^n = T_{\frac{n-1}{2}, \frac{n+1}{2}} = S^{\frac{n-1}{2}}(\sqrt{\frac{n-1}{2n}}) \times S^{\frac{n+1}{2}}(\sqrt{\frac{n+1}{2n}})$.
\end{prop}

We need the following four lemmas to prove this proposition. We omit the proof for the first two lemmas as they can be proved similarly to Lemma \ref{h}.

\begin{lem} \label{p}
Let $p(x) = \frac{(x+2)^{x+2}}{x^x (x+1)^2}$ be the function defined on $[1,+\infty)$, then $p(x)$ is an increasing function.
\end{lem}

\begin{lem} \label{q}
Let $q(x) = \sqrt{\frac{(x-2)^{x-2} (x+1)^{x+1}}{(x-1)^{x-3} x^{x+2}}}$ be the function defined on $[3,+\infty)$, then $q(x)$ is a decreasing function with limit $1$ at $+\infty$.
\end{lem}

\begin{lem} \label{+2 ineq}
$\text{Area}(T_{k,l}) < \text{Area}(T_{k-2,l+2})$ for $3 \leq k \leq l+1$.
\end{lem}

\begin{proof}

By \eqref{+2 ratio}, we know that 
\begin{align*}
\frac{\text{Area}(T_{k-2,l+2})}{\text{Area}(T_{k,l})} &= \frac{\text{Area}(S^{k-2}(\sqrt{\frac{k-2}{k+l}}) \times S^{l+2}(\sqrt{\frac{l+2}{k+l}})))}{\text{Area}(S^{k}(\sqrt{\frac{k}{k+l}}) \times S^{l}(\sqrt{\frac{l}{k+l}})))} \\
&= \sqrt{\frac{(k-2)^{k-2} (l+2)^{l+2}}{k^k l^l}} \cdot \frac{2\pi}{l+1} \cdot \frac{k-1}{2\pi} \\
&= \sqrt{\frac{p(l)}{p(k-2)}} >1.
\end{align*}
(by lemma \ref{p}, $p(x)$ is an increasing function, and $k-2 \leq l-1 < l$.)

\end{proof}

\begin{lem} \label{+1 ineq}
$\text{Area}(T_{k,k}) < \text{Area}(T_{k-1,k+1})$ for $k \geq 2$.
\end{lem}

\begin{proof}
By the surface area formula \eqref{surface area formula}, we can compute that 
$$ \frac{\text{Area}(T_{k-1,k+1})}{\text{Area}(T_{k,k})} = \left\{
\begin{aligned}
(\frac{(k+1) \cdot (k-1) \cdots 2}{k \cdot (k-2) \cdots 1})^2 \cdot \frac{2}{\pi} \sqrt{\frac{(k-1)^{k-1}(k+1)^{k-3}}{k^{2k-2}}} &    & \rm{k \ is \  odd}, \\
(\frac{(k+1) \cdot (k-1) \cdots 1}{k \cdot (k-2) \cdots 2})^2 \cdot \frac{\pi}{2} \sqrt{\frac{(k-1)^{k-1}(k+1)^{k-3}}{k^{2k-2}}} &     & \rm{k \ is \ even}.
\end{aligned}
\right.
$$

Let $c_k = \frac{\text{Area}(T_{k-1,k+1})}{\text{Area}(T_{k,k})}$, then one can check that $c_{k-1} c_{k} = q(k)$.

By Stirling's approximation, $n! \sim \sqrt{2\pi n} (\frac{n}{e})^n$, one can check that $\{c_n\}$ converges to $1$ as $n \to +\infty$.

By lemma \ref{q}, $q(x)$ is a decreasing function, thus we have the following inequality:
\begin{align*}
\frac{c_{k+2}}{c_k} = \frac{c_{k+2} c_{k+1}}{c_{k+1} c_k} = \frac{q(k+2)}{q(k+1)} < 1.
\end{align*}

Therefore $\{c_2, c_4, c_6, \cdots\}$, $\{c_3, c_5, c_7, \cdots\}$ are two decreasing sequences. Since $\{c_k\}$ converges to $1$ as $n \to +\infty$, thus we claim that $c_k > 1$ for all $k \geq 2$.

\end{proof}

\begin{proof}[Proof of Proposition\ref{min}]

We can easily see that $\text{Area}(T_{k,l}) = Area(T_{l,k})$.

When $n$ is even, $n = 2k$ for some positive integer $k \geq 2$.

By lemma \ref{+2 ineq}, we know that
\begin{align*}
\text{Area}(T_{k,k}) < \text{Area}(T_{k-2,k+2}) < \text{Area}(T_{k-4,k+4}) < \cdots \\
\text{Area}(T_{k-1,k+1}) < \text{Area}(T_{k-3,k+3}) < \text{Area}(T_{k-5,k+5}) < \cdots
\end{align*}

By lemma \ref{+1 ineq}, $\text{Area}(T_{k,k}) < \text{Area}(T_{k-1,k+1})$. Therefore $T_{min}^n = T_{k,k}$.

When $n$ is odd, $n = 2k+1$ for some positive integer $k$.

By lemma \ref{+2 ineq}, we know that
\begin{align*}
\text{Area}(T_{k,k+1}) &< \text{Area}(T_{k-2,k+3}) < \text{Area}(T_{k-4,k+5}) < \cdots \\
\text{Area}(T_{k-1,k+2}) &< \text{Area}(T_{k-3,k+4}) < \text{Area}(T_{k-5,k+6}) < \cdots \\
\text{Area}(T_{k,k+1}) &= \text{Area}(T_{k+1,k}) < \text{Area}(T_{k-1,k+2}).
\end{align*}

Therefore $T_{min}^n = T_{k,k+1}$.

\end{proof}

\section{Monotonicity of the ratio} \label{mono}

In this section, We prove that the ratio between the area of the minimal Clifford hypersurface and the area of the sphere is strictly decreasing. We expect to use this monotonicity formula and the standard dimension reduction argument to classify low area minimal hypersurface in $S^{n+1}$.

\begin{thm} \label{decreasing}
The sequence $\{ \frac{\text{Area}(T_{min}^n)}{\text{Area}(S^n)} \}_{n \geq 2}$ is a strictly decreasing sequence with limit $\sqrt{2}$ as $n \to +\infty$.
\end{thm} 

We omit the proof for the next two lemmas as they can be proved similarly to Lemma \ref{h}.

\begin{lem} \label{u}
Let $u(x) = 16 \sqrt{\frac{(x+2)^{x+2}(x+3)^{x+3}(2x+1)^{2x-1}}{x^x(x+1)^{x-3}(2x+3)^2(2x+5)^{2x+5}}}$ be the function defined on $[1,+\infty)$, then $u(x) > 1$.
\end{lem}

\begin{lem} \label{v}
Let $v(x) = \frac{1}{16} \sqrt{\frac{x^x(x+1)^{x+1}(2x+3)^2(2x+5)^{2x+7}}{(x+2)^{x+6}(x+3)^{x+3}(2x+1)^{2x+1}}}$ be the function defined on $[1,+\infty)$, then $v(x) > 1$.
\end{lem}

\begin{proof}[Proof of Theorem\ref{decreasing}]

Let $a_n = \frac{\text{Area}(T_{min}^n)}{\text{Area}(S^n)}$, then $a_2 = \frac{\pi}{2}, a_3 = \frac{8}{3\sqrt{3}}, a_4 = \frac{3}{2}, a_5 = \frac{48\sqrt{15}}{125}, a_6 = \frac{15\pi}{32}, a_7 = \frac{768\sqrt{21}}{2401}$, and we have $a_2 > a_3 > a_4 > a_5 > a_6 > a_7$.

By using the formula (\ref{+2 ratio}), it's not hard to check that for all positive integers $k$,
\begin{align*}
& a_{2k + 4} = \frac{(2k+1)(2k+3)}{4(k+1)^2}a_{2k = }\frac{4k^2 + 8k + 3}{4k^2 + 8k + 4} a_{2k} < a_{2k}, \\
& a_{2k+5} = 4\sqrt{\frac{(k+2)^{k+2}(k+3)^{k+3}(2k+1)^{2k+1}}{k^k(k+1)^{k+1}(2k+5)^{2k+5}}} a_{2k+1} < a_{2k+1}.
\end{align*}

Let $b_n=\frac{a_{n+1}}{a_n}$, by lemma \ref{u},\ref{v}, we know that
\begin{align*}
\frac{b_{2k+4}}{b_{2k}} &= \frac{a_{2k+5}}{a_{2k+4}} \cdot \frac{a_{2k}}{a_{2k+1}} = 16\sqrt{\frac{(k+2)^{k+2}(k+3)^{k+3}(2k+1)^{2k-1}}{k^k(k+1)^{k-3}(2k+3)^2(2k+5)^{2k+5}}} = u(k) > 1, \\
\frac{b_{2k+5}}{b_{2k+1}} &= \frac{a_{2k+6}}{a_{2k+5}} \cdot \frac{a_{2k+1}}{a_{2k+2}} = \frac{1}{16} \sqrt{\frac{k^k(k+1)^{k+1}(2k+3)^2(2k+5)^{2k+7}}{(k+2)^{k+6}(k+3)^{k+3}(2k+1)^{2k+1}}} =v(k) > 1.
\end{align*}

Thus the four subsequences $\{b_{i+4j}\}_{j = 1,2,3, \cdots}$ ($i = 2,3,4,5$) are increasing sequences.

Recall the Euler's product formula: \begin{align*}
\sin(\pi z) = \pi z \prod_{i = 1}^{\infty} (1 - \frac{z^2}{i^2}).
\end{align*}

Let $z = \frac{1}{2}, \frac{1}{4}$, we have the following equality:
\begin{align*}
\prod_{i = 1}^{\infty} (1 - \frac{1}{4i^2}) &= \frac{2}{\pi},\\
\prod_{i = 1}^{\infty} (1 - \frac{1}{16i^2}) &= \frac{2\sqrt{2}}{\pi}, \\
\prod_{i = 1}^{\infty} (1 - \frac{1}{4(2i+1)^2}) &= \frac{\prod\limits_{i = 1}^{\infty} (1 - \frac{1}{4i^2})}{(1-\frac{1}{4})\prod\limits_{i = 1}^{\infty} (1 - \frac{1}{16i^2})} =  \frac{2\sqrt{2}}{3}.
\end{align*}

One can check that \begin{align*}
a_{2k+4j+1} = 4^j \sqrt{\frac{(k+2j)^{k+2j}(k+2j+1)^{k+2j+1}(2k+1)^{2k+1}}{k^k(k+1)^{k+1}(2k+4j+1)^{2k+4j+1}}} a_{2k+1}    
\end{align*}
for positive integer $j$.

It's not hard to see that \begin{align*}
\lim\limits_{j \to \infty} 4^j \sqrt{\frac{(k+2j)^{k+2j}(k+2j+1)^{k+2j+1}(2k+1)^{2k+1}}{k^k(k+1)^{k+1}(2k+4j+1)^{2k+4j+1}}} = \sqrt{\frac{(2k+1)^{2k+1}}{k^k(k+1)^{k+1}2^{2k+1}}}.
\end{align*}

Therefore we have the following limits:
\begin{align*}
\lim\limits_{i \to \infty} a_{2+4i} &= a_2 \cdot \prod\limits_{i = 1}^{\infty} (1 - \frac{1}{16i^2}) = \sqrt{2}, \\ 
\lim\limits_{i \to \infty} a_{3+4i} &= a_3 \cdot \sqrt{\frac{3^3}{1^1 \cdot 2^2 \cdot 2^3}} =  \sqrt{2}, \\ 
\lim\limits_{i \to \infty} a_{4+4i} &=  a_4 \cdot \prod_{i = 1}^{\infty} (1 - \frac{1}{4(2i+1)^2}) = \sqrt{2}, \\
\lim\limits_{i \to \infty} a_{5+4i} &= a_5 \cdot \sqrt{\frac{5^5}{2^2 \cdot 3^3 \cdot 2^5}} = \sqrt{2}.
\end{align*}

Hence we know that the sequence $\{a_n\}$ converges to $\sqrt{2}$, and the sequence $\{b_n\}$ converges to $1$ (since $b_n = \frac{a_{n+1}}{a_n}$).

Since for $i = 2,3,4,5$, $\{b_{i+4j}\}_{j = 1,2,3,\cdots}$ is an increasing sequence, we claim that $b_n < 1$ for all $n$, and therefore the sequence $\{a_n\}_{n \geq 2}$ is a decreasing sequence.

\end{proof}

We have the following area lower bound for Clifford hypersurface:

\begin{coro} \label{CH lower bound}
All $n-$dimensional Clifford hypersurfaces have area at least $\sqrt{2} \text{Area}(S^n)$.
\end{coro}

\bibliography{main}
\bibliographystyle{acm}

\end{document}